\documentclass[11pt]{amsart}

\usepackage{amsthm,amsmath,amsfonts,amssymb,amscd,mathrsfs, diagrams}
\usepackage{txfonts}%,pxfonts}
\usepackage{supertabular,soul}
\usepackage{tikz, graphicx,color}
 \usepackage{multirow}
\usepackage[pdftex,
            pdfauthor={Francis},
            pdftitle={A Survey of FJRW Theory},
            pdfsubject={FJRW theory, Mirror symmetry},
            pdfkeywords={FJRW, Mirror symmetry}]{hyperref}

\usepackage{bbm} %\mathbb numbers and other symbols
\usepackage{hyperref}
\usepackage{yfonts}
\usepackage{eucal}

\newcommand{\comment}[1]{}

\newcommand{\g}{\mathbf{g}}
\newcommand{\h}{\mathbf{h}}
\newcommand{\rhot}{\rho_{\operatorname{tree}}}
\newcommand{\rhol}{\rho_{\operatorname{loop}}}

\newcommand{\bgamma}{\boldsymbol{\gamma}}
\newcommand{\Lcal}{\sL}
\newcommand{\Pcal}{\sP}
\newcommand{\Ccal}{\sC}
\newcommand{\kk}{\omega} % the canonical bundle
\newcommand{\pkk}{\mathring{\omega}}  %principal bundle assoc to omega
\newcommand{\klogc}{\kk_{\log,\Ccal}}    %the log canonical bundle on C
\newcommand{\pklogc}{\pkk_{\log,\Ccal}}    %principle bundle assoc to log canonical bundle on C
\newcommand{\spn}{\varkappa}  %isom of \zeta_*\Pcal to \pklogc

\newcommand{\chiR}{\zeta}
\newcommand{\genj}{\langle J \rangle}
\newcommand{\defital}{\textit}

\newcommand{\Z}{\mathbb Z}

\newcommand{\C}{\mathbb C}
\newcommand{\CC}{\mathbb C}

\newcommand{\Q}{\mathbb Q}
\newcommand{\fjrw}[2]{ \left\lceil #1 \:; #2 \right\rfloor }

\newcommand{\zero}{\mathbf 0}

\newcommand{\Gmax}{G_{\max}}

\newcommand{\Hess}{\operatorname{Hess}}

\newcommand{\one}{\mb{1}}

\newcommand{\B}{B}

\newcommand{\age}{\operatorname{age}}

\newcommand{\Span}{\text{Span}}

\newcommand{\mb}{\mathbf}

\newcommand{\bx}{\mathbf{x}}

\DeclareMathOperator{\aut}{Aut}

\newcommand{\sH}{\mathscr{H}}

\newcommand{\sC}{\mathscr{C}}

\newcommand{\sW}{\mathscr{W}}
\newcommand{\sQ}{\mathscr{Q}}
\newcommand{\sP}{\mathscr{P}}
\newcommand{\mgkbar}{\bMgn{g}{k}}

\newcommand{\bMgn}[2]{\overline{\mathscr{M}}_{#1,#2}}

\newcommand{\sL}{\mathscr{L}}
\renewcommand{\Re}{\mathfrak{Re}}

\newtheorem{thm}{Theorem}[section]

\newtheorem{prop}[thm]{Proposition}

\newtheorem{conj}[thm]{Conjecture}

\theoremstyle{definition}
\newtheorem{remark}[thm]{Remark}

\newtheorem{defn}[thm]{Definition}

\newtheorem{example}[thm]{Example}

\theoremstyle{remark}

\DeclareMathOperator{\fix}{Fix}
\providecommand*{\propertyautorefname}{Property}

\let\oldmarginpar\marginpar
\renewcommand\marginpar[1]{\oldmarginpar[\raggedleft\footnotesize #1]%
{\raggedright\footnotesize #1}}

\begin{document}

\date{\today}

\title[Survey of  FJRW Theory]{A Brief Survey of FJRW Theory}
\author{Amanda E. Francis}
\address{Department of Mathematics, Brigham Young University, Provo, UT 84602, USA}
\email{amanda@mathematics.byu.edu}
\author{Tyler J. Jarvis}
\address{Department of Mathematics, Brigham Young University, Provo, UT 84602, USA}
\email{jarvis@mathematics.byu.edu}

%\subjclass[2010]{Primary: 14N35, 53D45, Secondary: 32S05, 37K10, 37K20, 35Q53}

\begin{abstract}
In this paper we describe some of the constructions of FJRW theory.  We also briefly describe its relation to Saito-Givental theory via Landau-Ginzburg mirror symmetry and its relation to Gromov-Witten theory via the Landau-Ginzburg/Calabi-Yau correspondence.  We conclude with a discussion of some of the recent results in the field.
\end{abstract}

\maketitle

%\setcounter{tocdepth}{1}
%\tableofcontents

\section{Introduction}

In this article we briefly describe what is now called \emph{Fan-Jarvis-Ruan-Witten (FJRW) theory}.  This is analogous to Gromov-Witten (GW) theory in many ways.  
It associates a cohomological field theory (and hence also Frobenius manifold) to each pair $W,G$ of a nondegenerate quasihomogeneous polynomial $W$ and an Abelian group $G$ of symmetries of $W$.  

Some of the many similarities between FJRW theory and GW theory include the fact that the state space of each is the cohomology of a ``target space'' (see Section \ref{sec_statespace}), both are equipped with stabilization and evaluation maps, and both have virtual cycles that satisfy certain properties (see Section \ref{sec_virtcycle}) which facilitate computation and that, when pushed down to $\bMgn{g}{k}$, give a cohomological field theory satisfying  the WDVV equation, the string equation, the dilaton equation, and topological recursion relations.  These similarities are not just superficial, as the two theories are highly related via the LG/CY correspondence.

FJRW theory also plays a role in mirror symmetry.  In particular, it provides a Landau-Ginzburg A model for the polynomial $W$, just as Gromov-Witten theory provides a Calabi-Yau A model.  We give a brief summary of mirror symmetry and the LG/CY correspondence here.  We conclude with a brief discussion of recent progress on a mathematical approach to the Gauged Linear Sigma Model, which generalizes FJRW theory and the LG/CY correspondence to complete intersections, toric varieties, and more general spaces.

\subsection{Landau-Ginzburg Mirror Symmetry}

To any isolated singularity, we may construct the  local algebra (or Milnor ring) of $W$, which is a Frobenius algebra, and for each choice of primitive form, K.~Saito constructed a Frobenius manifold deforming the Milnor ring \cite{Sai1,Sai2,Sai3,SaiTak}, 
which may be thought of as the 
genus-zero Landau-Ginzburg B model corresponding to $W.$  Since the Saito
construction is semisimple, Givental's 
theory of ``higher genus Frobenius manifolds'' \cite{Giv:quantization, givental} determines a potential function for the theory.

The \emph{Landau-Ginzburg (LG) Mirror Symmetry Conjecture} predicts
that for a large class of polynomials $W$ (called invertible) with a group $G$ of admissible symmetries
of $W$, there is  a dual polynomial ${W^T}$ and dual group ${G^T}$ of
symmetries of ${W^T}$ such that FJRW Landau-Ginzburg A model for the pair $(W,G)$
is isomorphic to the
Saito-Givental B model construction for the pair $(W^T,G^T)$.  There are also some other formulations of Landau-Ginzburg mirror symmetry which will discuss in Section~\ref{sec:mirror}, along with a brief survey of recent results.

\subsection{Landau-Ginzburg/Calabi-Yau Correspondence}

The inputs to FJRW theory are a quasihomogeneous polynomial $W$ with an isolated singularity at
the origin and an Abelian group $G$ of symmetries of $W$.  But such a $W$ also defines a smooth hypersurface $X_W = \{W=0\}$ in weighted
projective space.  And if we let
$J = G \cap \CC^*_R$, where $\CC^*_R$ is the 1-parameter  group
defining weighted projective space as $\mathbb{P}(c_1,\dots,c_N) = [\C^N/\C^*_R]$, then there is an induced
action of the group $\tilde{G} = G/\genj$ on $X_W$.  

In the case that the quotient
orbifold $[X_W/{\tilde{G}}]$ is Calabi-Yau, the
\emph{Calabi-Yau/Landau-Ginzburg (CY/LG) Correspondence} predicts that
the analytic continuation of the FJRW potential of the pair
$(W, G)$, after a suitable  
symplectic transformation, will match
precisely with the orbifold Gromov-Witten potential of the
orbifold $[X_W/\tilde{G}]$ in all genera.
We will discuss the progress that has been made on this conjecture (and some generalizations) in Section~\ref{sec:LG-CY}.

Combining  mirror symmetry for both Calabi-Yau and Landau-Ginzburg theories with the CY/LG correspondence, we have the following (partly conjectural, partly proven)  picture:

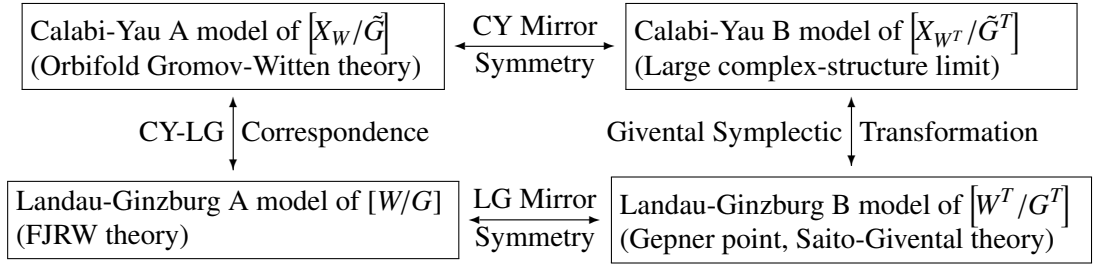
\newarrow{Corresp} <--->
%\newarrow{Isom}{}==={}
\begin{equation*}\label{diag:CYLG}
\begin{diagram}
\text{\fbox{\parbox{14em}{Calabi-Yau A model  of $\left[\!X_W/\tilde{G}\!\right]$\\ (Orbifold Gromov-Witten theory)}}} 		&      \rCorresp^{\text{CY Mirror}}_{\text{Symmetry}}		& \text{\fbox{\parbox{15em}{Calabi-Yau B model  of $\left[\!X_{W^T}/\tilde{G}^T\!\right]$ \\ (Large complex-structure limit)}}}\\
\dCorresp^{\text{CY-LG}}_{\text{Correspondence}}				&
& \dCorresp^{\text{{Givental Symplectic}}}_{\text{{Transformation}}}\\
\text{\fbox{\parbox{15em}{Landau-Ginzburg A model  of $\left[W/G\right]$\\ (FJRW theory)}}} & 	\rCorresp^{\text{LG Mirror}}_{\text{Symmetry}}		&\text{\fbox{\parbox{16em}{Landau-Ginzburg B model of $\left[\!W^T/G^T\!\right]$ \\  (Gepner point, Saito-Givental theory) }}}
\end{diagram}
\end{equation*}

\subsection{Relations on the Stack of Stable Curves}

In addition to fitting nicely into the mirror-symmetry picture above, FJRW theory can also be used to identify relations in the cohomology of the the stack of stable curves.  

The theory of $3$-spin curves corresponds to FJRW theory for the  polynomial $W = x^3$ with the group $G = \mu_3$.  Pandharipande-Pixton-Zvonkine \cite{PPZ} used the grading on the cohomological field theory (See Section~\ref{sec:cohft} for a definition) of $3$-spin curves to identify new cohomological relations on the stack of stable curves.  Higher spin curves give additional relations \cite{PPZ2}, and Janda \cite{Janda} has shown those follow from similar relations arising from the Gromov-Witten cohomological field theory for projective space. 

It is natural to conjecture that applying the Pandharipande-Pixton-Zvonkine methods to the FJRW theory of more general polynomials and groups would give additional relations, although to our knowledge, this has not yet been explored.

\section{Cohomological Field Theories}\label{sec:cohft}

Given the data of a quasihomogeneous polynomial and a finite group of diagonal symmetries, FJRW theory produces a cohomological field theory, in the sense of Kontsevich and Manin \cite{KontMan}.
We begin with a brief review of cohomological field theories.  These are essentially collections of cohomology classes that behave well with respect to the gluing maps of pointed stable curves.

\begin{defn}
For any nonnegative integers $g,k$ with $2g-2+k >0$ let $\bMgn{g}{k}$ denote the stack of stable $k$-pointed curves of genus $g$.  Given nonnegative integers $g_1, g_2, k_1,k_2$ with $2g_i-1+k_i>0$ for $i\in \{1,2\}$, we define a gluing map 
\[
\rhot:\bMgn{g_1}{k_1+1}\times \bMgn{g_2}{k_2+1} \to \bMgn{g_1+g_2}{k_1+k_2}
\]
by attaching the two additional marked points to form a node, as in Figure~\ref{fig:gluetree}.
Similarly, for any nonnegative integers $g,k$ with $2g+k>0$ we define a gluing map 
\[
\rhol:\bMgn{g}{k+2} \to \bMgn{g+1}{k},
\]
again by attaching the two additional marked points to form a node, as in Figure~\ref{fig:glueloop}.

Finally, if $2g-2+k>0$ we define a \emph{forgetting tails} map 
\[
\tau: \bMgn{g}{k+1} \to \bMgn{g}{k}
\]
by forgetting the last marked point and successively contracting any resulting unstable components.
\end{defn}

\begin{figure}[htb]
%\begin{center}
\begin{tikzpicture}
\draw[domain = -.4:.5] plot (\x-4,-.5+1.8*\x*sqrt(\x+1);
\draw[domain = -.4:.5] plot (\x-6,-1.8*\x*sqrt(\x+1);
%\node at (-5,0) {$\times$};
\node at (-4,-.5) {$\bullet$};
\node at (-6,0) {$\bullet$};

\draw[->] (-2.5,0) -- (-.5,0);
\node at (-1.5,.4) {$\rhot$};

\draw[domain = -.4:.5] plot (\x+1,1.8*\x*sqrt(\x+1);
\draw[domain = -.4:.5] plot (\x+1,-1.8*\x*sqrt(\x+1);
\node at (1,0) {$\bullet$};
\end{tikzpicture}
%\end{center}
\caption{The morphism $\rhot$ glues a marked point on each of two stable curves to form a new, nodal curve.\label{fig:gluetree} }
\end{figure}
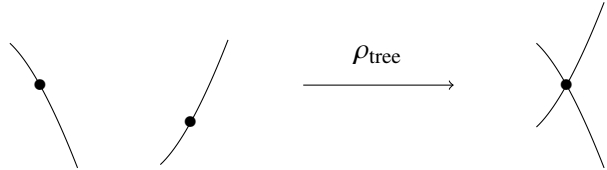

%\begin{center}
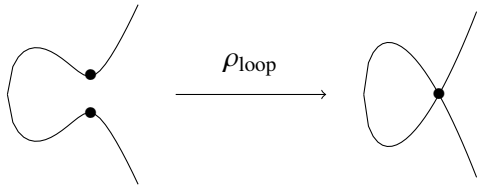
\begin{figure}[htb]

\begin{tikzpicture}
\draw[domain = -1.475:2] plot (.5*\x-4,-.03+.5*sqrt(\x^3-1.5*\x+1);
\draw[domain = -1.475:2] plot (.5*\x-4,.03-.5*sqrt(\x^3-1.5*\x+1);
\node at (-3.63,-.25) {$\bullet$};
\node at (-3.63,.25) {$\bullet$};

\draw[->] (-2.5,0) -- (-.5,0);
\node at (-1.5,.4) {$\rhol$};

\draw[domain = -1:.5] plot (\x+1,1.8*\x*sqrt(\x+1);
\draw[domain = -1:.5] plot (\x+1,-1.8*\x*sqrt(\x+1);
\node at (1,0) {$\bullet$};
\end{tikzpicture}
%\end{center}
\caption{The morphism $\rhol$ glues a two marked points on one stable curve to form a new, nodal curve.\label{fig:glueloop}} 
\end{figure}

%\begin{figure}
%%\begin{center}
%\begin{tikzpicture}
%\draw[domain = -1.475:2] plot (.5*\x+1,.5*sqrt(\x^3-1.5*\x+1);
%\draw[->] (-2.5,0) -- (-.5,0);
%\node at (-1.5,.4) {$\tau$};
%\draw[domain = -1.475:2] plot (.5*\x-4,.5*sqrt(\x^3-1.5*\x+1);
%\node at (-4,.5) {$\bullet$};
%\end{tikzpicture}
%%\end{center}
%\end{figure}
%  

\begin{defn}
A \emph{cohomological field theory with flat identity} (abbreviated CohFT) consists of the data of
\begin{enumerate}
\item A vector space $\sH$ (called the \defital{state space}). 
\item A nondegenerate pairing: 
\[
\langle \ ,\ \rangle: \sH \otimes \sH  \rightarrow \C,
\]
\item An element $\one\in \sH$.

\item For each $g$, $k$ a $k$-linear form $\Lambda_{g,k}^{W,G} : \sH^k \rightarrow H^*(\bMgn{g}{k})$, satisfying the following axioms
	\begin{enumerate}
	\item  Each form
$\Lambda_{g,k}$ is equivariant under the action of the symmetric
  group $S_k$.

	\item For any $\alpha_1,\dots, \alpha_{k_1+k+2} \in \sH$ we have 
\[
\rhot^* \Lambda_{g_1+g_2,k_1+k_2}^{W,G}(\alpha_1, \ldots, \alpha_{k_1+k_2}) \\
= \sum_{\sigma, \tau} 
 \Lambda_{g_1,k_1+1}^{W,G}(\alpha_1, \ldots, \alpha_{k_1}, \sigma)\eta^{\sigma \tau}  \Lambda_{g_2,k_2+1}^{W,G}(\tau,\alpha_{k_1+1}, \ldots, \alpha_{k_1+k_2}),\]
 where $\sigma,\tau$ run over a basis of $\sH$ and $\eta^{\sigma,\tau}$ is the inverse of the matrix expressing the pairing in terms of this basis.
	\item For any $\alpha_1,\dots, \alpha_{k} \in \sH$ we have 
\[
\rhol^* \Lambda_{g+1,k}^{W,G}(\alpha_1, \ldots, \alpha_{k}) = 
 \sum_{\sigma, \tau} 
 \Lambda_{g,k+2}^{W,G}(\alpha_1, \ldots, \alpha_{k}, \sigma,\tau)\eta^{\sigma \tau} 
 \]
 where again $\sigma,\tau$ run over a basis of $\sH$ and $\eta^{\sigma,\tau}$ is the inverse of the matrix expressing the pairing in terms of this basis.
 \item For any $\alpha_1,\dots,\alpha_k\in \sH$ we have 
\[
\tau^* \Lambda_{g,k}^{W,G}(\alpha_1, \ldots, \alpha_{k}) =
 \Lambda_{g,k+1}^{W,G}(\alpha_1, \ldots, \alpha_{k}, \one)
 \]
  \item 
$\int_{\bMgn{0}{3}} \Lambda_{0,3}^{W,G}(\alpha_1, \alpha_2, \one) =
 \langle \alpha_1,\alpha_2 \rangle$
\end{enumerate}
\end{enumerate}
\end{defn}
Some of the most important examples of CohFTs are the Gromov-Witten theory of a variety $X$ and the FJRW theory of a quasihomogeneous polynomial. 

%\hl{Maybe this next stuff is unnecessary?}
%
%For any CohFT, and any $k$-tuple $(\alpha_1,\dots, \alpha_k)$ of elements of $\sH$, integrating the corresponding cohomology class $\int_{\bMgn{g}{k}} \Lambda_{g,k}^{W,G}(\alpha_1,\dots,\alpha_k)$ gives a correlator $\langle \alpha_1,\dots, \alpha_k\rangle_g$, and these can be assembled into a potential function in the usual way.
%
%The genus-zero three-point correlators of a CohFT define an associative product $\star$ on the the state space $\sH$, and the product satisfies the Frobenius property $\langle a \star b, c \rangle = \langle a ,b\star  c \rangle$ for all $a,b,c\in \sH$.  That is, $\sH,\star$ is a Frobenius algebra.  In the the case of Gromov-Witten theory, this Frobenius algebra called \emph{quantum cohomology}.   The potential function of all the the genus-zero correlators defines a Frobenius manifold (a family of Frobenius algebras).
%

\section{FJRW Theory}

 \subsection{Inputs}
The ingredients for FJRW theory are an admissible quasihomogeneous polynomial and an admissible group of diagonal symmetries.  Let us now define what we mean by these.

\begin{defn}
A polynomial $W \in \mathbb{C}[x_1, \ldots, x_n]$, where $W = \sum_i b_i\prod_{j = 1}^n x_j^{a_{i,j}}$ 
is called \defital{quasihomogeneous} if there exist positive rational numbers $q_j$ (called weights) for each variable $x_j$ such that each monomial of $W$ has weighted degree one.  That is, for every $i$ where $b_i \neq 0$, 
\[
\sum_j q_j a_{i,j} = 1.
\]

A polynomial $W \in \mathbb{C}[x_1, \ldots, x_n]$ is called \defital{nondegenerate} if it has an isolated singularity at the origin, otherwise it is called \defital{degenerate}.

Any polynomial which is both quasihomogeneous and nondegenerate will be considered \defital{admissible} if there are no monomials of $W$ of the form $x_ix_j$.  This is equivalent to saying that the weights $q_1,\dots, q_n$ are uniquely determined and all weights lie in the interval $(0,1/2)\cap\Q$.

We say that an admissible polynomial is \defital{invertible} if it has the same number of variables and monomials. 
\end{defn}
\begin{remark}\label{rem:integerweights}
It is often useful to use integer weights $c_1,\dots, c_N,d$ for $W$ such that $\gcd(c_1,\dots, c_N,d) =1$ and each $q_i = c_i/d$.
\end{remark}
\begin{example}
One may easily check that $x^3+y^3$ is admissible and invertible, $x^3 + y^3 + xy^2$ is admissible, $x^3 + y^3 +xy^4$ is not quasihomogeneous and $x^4 + x^2y^2$ is quasihomogeneous but degenerate.
\end{example}

\begin{remark}\label{rem:classification}
Classical singularity theory studies singularities up to so-called \defital{right equivalence}; that is, up to a smooth change of variables. The quantum singularity theory of FJRW theory appears to be more rigid, by comparison. The only changes of variables allowed in this theory are rescaling and reordering (relabeling)  variables of the same weight. Therefore, we often need classification results that are stronger than those of classical singularity theory.  On the other hand, the quantum invariants of FJRW theory are known to be independent of the polynomial $W$  and depend only on the weights and an \defital{admissible group} $G$, as we will see in Theorem \ref{thm_WG}.
\end{remark}
 
%%Maybe we want to cut this invertible stuff?  
%%TJ It seems ok to me.
A fundamental tool for studying Landau-Ginzburg mirror symmetry is the following classification of invertible singularities. 
\begin{prop} \cite{KrSk} \label{atomictypes}
A polynomial is (admissible and) invertible if, and only if, it can be written (up to relabeling, rearrangement, and rescaling of the variables) as a disjoint sum of invertible polynomials of one or more of the following three \defital{atomic types}:
\begin{align*}
  \text{Fermat {type:} } & x^a,\notag\\
  \text{loop   {type:} } & x_1^{a_1}x_2 + x_2^{a_2}x_3 + \dots + x_N^{a_N}x_1,\\
   \text{chain  {type:} } & x_1^{a_1}x_2 + x_2^{a_2}x_3 + \dots + x_N^{a_N},\notag
\end{align*}
where $a_k\geq 2$ for all $k\in \{1,\dots,N\}$.
\end{prop}

The other main ingredient of FJRW theory is the symmetry group.   Indeed, the theory depends heavily on the choice of symmetry group.   In this sense, it may be thought of as an orbifold singularity or the orbifold Landau-Ginzburg theory of $[W/G]$.

\begin{defn}\label{G}
Let $W$ be an admissible polynomial. The \defital{maximal diagonal symmetry group} $\Gmax$ is the subgroup of $(\CC^*)^N$ of elements that fix $W$:
\[
\Gmax = \{\ (\beta_1,\dots, \beta_N) \mid  W(\beta_1 x_1, \dots, \beta_Nx_N) = W(x_1,\dots, x_N)\ \}.
\]
It is easy to see that every element $\gamma \in \Gmax$ can be written in the form $\gamma = (e^{2\pi i \theta_1}, \ldots, e^{2\pi i \theta_N})$ with each $\theta_i \in [0,1)$.  We call the rational numbers $\theta_1, \dots, \theta_N$ the \emph{phases} of $\gamma$.  The sum of the phases is the \emph{age} of $\gamma$:
\[
\age(\gamma) = \sum_{i=1}^N \theta_i.
\]
\end{defn}

\begin{defn}\label{J_elem}
If $q_1, q_2, \ldots, q_N$ are the weights of $W$, then the element $J = (e^{2\pi i q_1}, \ldots, e^{2\pi i q_N})$ is an element of $\Gmax$, called the \defital{grading element}.  

Any subgroup $G$ of $\Gmax$ which contains the element $J$ is called \defital{admissible}.  
\end{defn}

\begin{defn}
%We use the notation $I_g = \{i \, \mid \,  g \cdot x_i = x_i\}$ to denote the set of indices of those variables fixed by an element $g$.  
Given an admissible group $G$ and an element $g\in G$, we let $\fix(g)$ denote the subspace of $\C^n$ which is fixed by $g$
\[
\fix(g) =\{\ (a_1, \ldots, a_N) \mid g(a_1,\dots,a_N) = (a_1,\dots, a_N) \},
\]
we denote the fixed indices by $I_g = \{ i \mid \theta^g_i =0\}$, and 
we write  
 $N_g = \# I_g = \dim(\fix(g))$.

Finally, we write 
\[
W_g = W|_{\fix(g)}
\] 
to  denote the polynomial $W$ restricted to $\fix(g)$.  
\end{defn}

Notice that 
\[
\age(g)+\age(g^{-1}) =\sum_{i: \theta_i^g \neq 0}1 = N - N_g
\]
 
 \begin{example}\label{example1}
Consider $W = x^3 + y^3+yz^2$.  This is a sum of a Fermat and a (reverse) chain polynomial.  In this case $q_x = q_y =q_z= \frac{1}{3}$, so  $J = (e^{2\pi i /3},e^{2\pi i /3}, e^{2\pi i /3})$ has order 3.  However, $\Gmax = \langle (e^{2\pi i /3},1,1), (1,e^{2\pi i /3}, e^{2\pi i \cdot 5/6}) \rangle$, which is a group of order 18.  It is easy to check that there are also admissible subgroups of order 6 and 9.  
 \end{example}
 
Each admissible symmetry group could potentially give a distinct FJRW theory, although there are cases where different groups do give the same theory.

\subsection{State Space}\label{sec_statespace}
In this section we define the state space of the theory for a given admissible polynomial $W\in \CC[x_1,\dots, x_N]$ and an admissible group $G$ of symmetries of $W$.

\begin{defn}
Because $W$ is $G$-invariant, it defines a function $W: \left[\C^N/G\right] \rightarrow \C$, where $[\C^N/G]$ denotes the stack quotient of $\C^N$ by the action of $G$.
We define 
\[
W^\infty = \left(\Re(W)\right)^{-1}(M,\infty)\text{ for }M >>0.
\]

Similarly, for $g \in G$, we have $W_g: \left[ \fix(g)/G\right] \rightarrow \C,$ and we define
\[
W_g^\infty = \left(\Re(W_g)\right)^{-1}(M,\infty).
\]
\end{defn}
\begin{defn}
The FJRW state space corresponding to $W$ and $G$ is given by the relative Chen-Ruan cohomology:
\[
\sH_{W,G} = H^{\bullet + 2q}_{CR}\left(\left[\C^N/G\right],W^\infty\right)
=\bigoplus_g H^{\bullet+2q-2\age(g)} \left(\left[ \fix(g)/G\right],W_g^\infty\right)
=\bigoplus_g \sH_g, 
\]
where for each $g\in G$, the subspace $\sH_g$ is called the \emph{$g$-sector} of the state space.  In particular, the degree of each element is shifted by $2\age(g)-2q$, so if $\varkappa \in H^{m} \left(\left[ \fix(g)/G\right],W_g^\infty,\C\right)$, then the degree of $\varkappa$ in $\sH_{W,G}$ is $m-2q+2\age(g)$. 
\end{defn}

\begin{remark}\label{def:narrow}\hfill
\begin{enumerate}
\item Since $G$ is Abelian, we have $H^n([\fix{g}/G],W^\infty,\C) = H^n(\fix(g),W^\infty,\C)^G$.
\item We have  $H^n(\fix(g),W^\infty,\C) = 0$ except when $n=N_g= \dim(\fix(g))$, therefore every element of $\sH_g$ has degree $N_g - 2q + 2\age(g)$.
\item If $\fix(g) = \{\zero\}$, then $W_g = 0$ and 
\[\sH_g = H^{0} \left(\left[ \zero/G\right],\emptyset,\C \right)  = \C,
\]
and the elements of $\sH_g$ all have degree $2\age(g) -2q$.  We often write $\sH_g = \C_g$ instead of $\C$, to indicate which sector of the state space we are describing.  

The sectors where $\fix(g) = \{\zero\}$ are called $\emph{narrow}$ sectors, and the other sectors are called \emph{broad}.

\item If $g = J$ (the grading element of $W$), then the sector $\sH_J$ is narrow, and $\age(J) = q$, so  
\[
\sH_J = H^{0} \left(\left[ \zero/G\right],\C \right)  = \C
\]
 is supported only in degree zero.  We denote the element $1\in \C_J$  by $\one$.  This will be the flat identity of our theory 
\end{enumerate}
\end{remark}

We will often use the following alternative description of $\sH_{W,G}$ in terms of germs of holomorphic $N$ forms on $\C^N$.   
\begin{thm}[\cite{Wall1,Wall2}] %%TODO check on other refs
Let $\Omega^N$ be the germs of holomorphic $N$-forms on $\C^N$ near zero.  We have
\[
H^N(\C^N,W^\infty,\C) \cong \Omega^N/(dW \wedge\Omega^{N-1})
\]
as $G$-modules.
\end{thm}

Notice that
\[
\frac{\Omega^N}{dW \wedge\Omega^{N-1}} %\cong \frac{\{f \cdot dx_1\wedge \ldots\wedge dx_N\}}{\{dW \wedge h \cdot dx_1 \wedge \hat{dx_i} \wedge \ldots \wedge dx_N\}} 
\cong \frac{\C[x_1, \ldots, x_N]}{(\frac{\partial W}{\partial x_1}, \ldots, \frac{\partial W}{\partial x_N})} \cdot dx_1 \wedge \ldots \wedge dx_N,
\]
so if $\sQ_W = {\C[x_1, \ldots, x_N]}/{(\frac{\partial W}{\partial x_1}, \ldots, \frac{\partial W}{\partial x_N})}$ is the local algebra (a.k.a. Milnor ring) of $W$, and if $d\bx = dx_1\wedge\dots\wedge dx_N$ is the obvious volume form, then   
\[
\sH_1 = H^N(\C^N,W^\infty,\C)^G \cong \left(\Omega^{N}/(dW \wedge\Omega^{N-1})\right)^G \cong (\sQ_W d\bx)^G.\]
  And, more generally, if the fixed indices of $g$ are $\{i_1, \ldots, i_{N_g}\} = I_g$, we have $\sQ_{W_g} = \C[x_{i_1}, \ldots, x_{i_{N_g}}]/{(\frac{\partial W}{\partial x_{i_1}}, \ldots, \frac{\partial W}{\partial x_{i_{N_g}}})}$ and $d\bx_{\fix(g)} = dx_{i_1}\wedge\dots\wedge dx_{i_{N_g}}$, therefore  
\begin{equation*}%\label{statespacemilnor}
\sH_{W,G} = \bigoplus_g H^{N_g}([\fix(g)/G],W_g^\infty,\C) \cong  \bigoplus_g \left(\Omega^{N_g}/(dW_g \wedge\Omega^{N_g-1})\right)^G \cong \bigoplus_g (\sQ_{W_g} d\bx_{\fix(g)})^G.
\end{equation*}
\begin{defn}\label{def:define-fjrw-notation}
We often write elements of the sector $\sH_g$ as $\fjrw{\phi}{g}$, where $\phi$ is an element of $\sQ_{W_g}d\bx_{\fix(g)}$.
\end{defn}

\begin{remark}
There is a natural Hodge structure on the $J$-invariant part of the relative cohomology.  There is also a natural $\Q$-grading on the local algebra $\sQ_{W_g}$ defined by the weighted degree of monomials in $\sQ_{W_g}$, and this gives a $\Q$-grading on $\sQ_{W_g} d\bx_{\fix(g)}$ defined by giving the volume form $d\, \bx_{\fix(g)}$ degree $\sum_{j=1}^{N_g} q_{i_j}$.  So if $\phi = m\, d\bx_{\fix(g)}$, then $\deg(\phi) = \deg(m) + \sum_{j=1}^{N_g} q_{i_j}$.

Tracking the Hodge structure through Wall's isomorphism connects the Hodge grading on relative cohomology with the grading on  $\left(\sQ_{W_g} d\bx_{\fix(g)}\right)^{\langle J \rangle}$:
\[
H^{N-p,p}(\C^{N_g},W_g^\infty,\C)^{\langle J\rangle}  \cong \left\{ \phi\in \sQ_{W_g} d\bx_{\fix(g)}  \mid \deg(\phi) = p \right\}^{\langle J \rangle}.
\]
We use this to put a bigrading on the state space 
\[
\sH_{W,G}^{a,b} = \bigoplus_g H^{a+q-\age{g},\, b+q-\age{g}}(\C^{N_g},W_g^\infty,\C)^G
\]
so if  $\deg(\phi) = p$ in the usual grading on $\sQ_{W_g}d\bx_{\fix(g)}$, then the bidegree of $\fjrw{\phi}{g}$ is 
\[
(\deg_+(\fjrw{\phi}{g}), \deg_-(\fjrw{\phi}{g})) = (N_g - p - q +\age(g), p - q +\age(g)),
\]   
and we have 
\[
\deg(\fjrw{\phi}{g}) = \deg_+(\fjrw{\phi}{g}) +  \deg_-(\fjrw{\phi}{g}).
\]
Finally, we note that it is often useful to define the complex degree to be 
\[
\deg_\C(\fjrw{\phi}{g}) = \frac12 \deg(\fjrw{\phi}{g}).
\]
\end{remark}

\begin{remark}\label{rem_Milnor}The following are standard facts from classical singularity theory and classical geometry: %TODO  get ref? 
\begin{enumerate}
\item $\sQ_W$ is finite dimensional if and only if $W$ has an isolated singularity at the origin.
\item $\sQ_{f_1} = \sQ_{f_2}$ if and only if $f_1 = f_2 \circ h$ for some analytic isomorphism  $h : \C^N \rightarrow \C^N$.
\item $\dim \sQ_W$ is determined solely by the weights $q_i = \frac{c_i}{d}$.  In fact 
\[
\mu:=\dim \sQ_W = \prod \left( \frac{1}{q_i}-1\right)
\]
\item The monomials of highest degree in $\sQ_W$ have degree $\hat{c}$ where $\hat{c} = N-2q = \sum_{i=1}^N (1-2q_i)$ is called the \emph{central charge} of $W$.  The subspace of $\sQ_W$ of complex degree $\hat{c}$ is one-dimensional and is generated by $\operatorname{Hess}(W) = \det\left[ \frac{\partial^2 W}{\partial x_i \partial x_j}\right]$.
\item The hypersurface $X_W = \left\{W = 0\right\}\subset \C^N$ is Calabi-Yau, if and only if, $\hat{c} = N-2$. 
\end{enumerate}
\end{remark}

\begin{example}\label{examplestatespace}
Consider $W = x^3 + y^3 + y z^2$ as in Example \ref{example1} with the group $G = \langle \gamma \rangle$, where $\gamma = (e^{2\pi i /3},e^{2\pi i /3}, e^{2\pi i \cdot 5/6})$.  Note that in this case we have $J = \gamma^4$.

We can find a basis for the state space $\sH_{W,G} = \bigoplus_g \sH_g^G$, in the following way.

For $g \in  \gamma, \gamma^2, \gamma^4, \gamma^5$ we have $I_g = \emptyset$ and  \[
\sH_g = \Span \{ \fjrw{1}{g}\}\cong \C.\]  
Since the usual grading of $1\in \C$ is $0$, and since $N_g = 0$ and $q=1$, we have 
\[
(\deg_+\fjrw{1}{g},\deg_-\fjrw{1}{g})  =(0 - 0 - 1 + \age(g), 0 - 1 + \age(g))
\]
From this, we easily compute
\begin{align*}
(\deg_+\fjrw{1}{\gamma},\deg_-\fjrw{1}{\gamma}) & = (1/2, 1/2)\\
(\deg_+\fjrw{1}{\gamma^2},\deg_-\fjrw{1}{\gamma^2}) & = (1,1)\\
(\deg_+\fjrw{1}{\gamma^4},\deg_-\fjrw{1}{\gamma^4}) & = (0,0)\\
(\deg_+\fjrw{1}{\gamma^5},\deg_-\fjrw{1}{\gamma^5}) & = (1/2, 1/2)
\end{align*}

For $g = \gamma^3 = (1,1,-1)$ we are interested in those $x^i y^j dx\wedge dy $ in $\sQ_g dx\wedge dy$ that are fixed by $G$, that is, where 
\[
\gamma \cdot x^i y^j dx\wedge dy  = e^{2\pi i \cdot \frac{i+j+1+1}{3}}x^i y^j dx\wedge dy  = x^i y^j dx\wedge dy. 
\]
  So, $\sH_g = \Span\{ \fjrw{x}{g}, \fjrw{y}{g}\}\cong \C^2$.

Since the usual grading of both $x\,dx\wedge dy$ and $y\,dx\wedge dy$ in $\sQ_{W_g}dx\wedge dy$ is $1$, and since $N_g = 2$, since $\age(\gamma^3) = 1/2$, and since $q=1$, we have 
\[
(\deg_+\fjrw{x}{\gamma^3},\deg_-\fjrw{x}{\gamma^3}) = (\deg_+\fjrw{y}{\gamma^3},\deg_-\fjrw{y}{\gamma^3})   = (1/2,1/2).
\]

Finally, for $g =\gamma^0 = (1,1,1)$, it is easy to see that none of the monomials $x^i y^j z^k dx\wedge dy \wedge dz$ in $\sQ_g d\bx$  are fixed by $G$, so  $\sH_g = \emptyset$, and  
\[
\sH_{W,G} = \Span \left\{ 
\fjrw{1}{\gamma}, \fjrw{1}{\gamma^2},\fjrw{x}{\gamma^3}, \fjrw{y}{\gamma^3},\fjrw{1}{\gamma^4}, \fjrw{1}{\gamma^5}
\right\}.
\]
\end{example}

Although our definition of the state space depends \emph{a priori} on the polynomial $W$, as a bigraded vector space the state space really only depends on the weights $q_1,\dots, q_n$ and the group $G$ (thought of as a subgroup of $GL(N)$.)
\begin{thm}\label{thm_WG}
As a bigraded vector space, the state space is determined only by the weights $q_1,\dots q_N$ and by the action of the group $G$ on $\C^N$.
\end{thm}
\begin{proof} If there is a continuous family of $G$-invariant, quasihomogeneous polynomials $W_t:\C^N \to \C$, all with the same weights $q_1,\dots, q_N$, then for each $t_1,t_2$ we have 
$H^\bullet(\C^N,W^{\infty}_{t_1},\C)^G \cong  H^\bullet(\C^N,W^{\infty}_{t_2},\C)^G$.  

The result now follows from the fact that the space of $G$-invariant admissible polynomials with a given set of weights $q_1,\dots, q_N$ is path connected.  This can be seen from the fact that the locus of polynomials with singularities away from the origin is a subvariety of complex codimension at least $1$ in the space of $G$-invariant quasihomogeneous polynomials. 
\end{proof} 
\begin{remark}
As an alternative (algebraic) proof to the previous theorem, one can consider the representation ring $\operatorname{Rep}(G)$ of $G$ and the $\operatorname{Rep}(G)$-valued Poincar\'e polynomial $P(t)$ of the $G$-module $\Omega^N/(dW \wedge\Omega^{N-1})$. The Koszul resolution gives us
\[
P(t) = \prod_{i=1}^N \frac{\rho_i - t^{d-q_i}}{1-\rho_it^{q_i}},
\]
where each $\rho_i$ is the representation of $G$ on the span of the variable $x_i$.
Since this expression for $P(t)$ is obviously independent of $W$, the result follows.  The details are given in \cite{Tay}.
\end{remark}

\subsection{Pairing}\label{sec_pairing}
Next we define a nondegenerate pairing, as required for the construction of our CohFT.   If $W^{-\infty}$ is defined to be 
\[
W^\infty = \left(\Re(W)\right)^{-1}(-\infty,-M)\text{ for }M >>0,
\]
then there is a natural perfect pairing
\[
\langle\ ,  \rangle: H_N(\C^N, W^{-\infty}, \Z)\otimes H_N(\C^N, W^{\infty}, \Z)\rightarrow \Z
\]
defined by intersecting the relative homology cycles. But we need a pairing on $H^N(\C^N,  W^{\infty}, \C)$.

Choose any $\xi$ such that $\xi^d=-1$. Multiplication by the diagonal matrix $(\xi^{n_1}, \dots, \xi^{n_N})$ defines a map
    $I : \C^N\rTo \C^N$
    sending $W^{\infty}\rTo W^{-\infty}$. Hence, it induces an isomorphism
\[
I_*: H_N(\C^N, W^{\infty}, \C)\rTo H_N(\C^N, W^{-\infty},
    \C).
 \]
We define a pairing on $H_N(\C^N, W^{\infty}, \Z)$ by
\[
\langle\Delta_i, \Delta_j\rangle=\langle \Delta_i, I_*(\Delta_j)\rangle.
\]
This induces a pairing (denoted by $\langle \ , \ \rangle$) on the dual space  $H^N(\C^N, W^{\infty}, \C)$.  Changing the choice of $\xi$ will change the isomorphism $I$ by an element of the group $\langle J \rangle$, and $I^2 \in \langle J \rangle$.
Therefore, the pairing is independent of the choice of $I$ on the
invariant subspace $H^N(\C^N,W^{\infty}, \C)^{\langle J \rangle}$ or on
$H^N(\C^N,W^{\infty}, \C)^G$ for any admissible group $G$.

Carefully tracking this pairing through Wall's isomorphism shows that it is equivalent to the residue pairing $\langle \ , \ \rangle_W$ on the Milnor ring, which can be computed by the following equation:
\[
f_1 \cdot f_2 = \frac{\langle f_1, f_2 \rangle_W}{\Hess(W)} + \text{lower order terms}.
\]

Now we define the pairing on the orbifolded state space $\sH_{W,G} = \bigoplus_g \sH_g^G$ by noting that there is a natural isomorphism $\beta_g:\sH_g \cong \sH_{g^{-1}}$, so we may define 
$
\langle \  ,\ \rangle: \sH_{W,G} \times \sH_{W,G} \rightarrow \C$ by 
\[
\langle \fjrw{f_1}{g_1}, \fjrw{f_2}{ g_2} \rangle = 
\begin{cases}
\langle f_1, \beta_{g_1}^{-1} f_2 \rangle_{W_{g_1}} & g_1 = g_2^{-1}\\
                                 0 & g_1 \neq g_2^{-1},
\end{cases}
\]
and extending linearly.

Note that the pairing has bigraded degree $(\hat{c},\hat{c})$, so that for any $p_1,p_2$ we have:
\[
\langle \ ,\ \rangle: \sH_{W,G}^{\hat{c}-p_1,\hat{c} - p_2 } \otimes \sH_{W,G}^{p_1,p_2} \rightarrow \C.
\]

\subsection{The Moduli Space}

A key part of FJRW theory is the stack $\sW_{g,k}^{W,G}$ of stable W-curves, which has a finite morphism to $\mgkbar$.  We define a virtual cycle on $\sW_{g,k}^{W,G}$, and the CohFT classes $\Lambda_{g,k}^{W,G}$ will be defined as the pushforward to $\mgkbar$ of the Poincar\'e dual of the virtual cycle.

Roughly speaking $\sW_{g,k}^{W,G}$ is the stack of stable orbicurves\footnote{For a more thorough discussion of orbicurves and orbifold line bundles, see \cite[\S2.1]{FJR}} with one orbifold line bundle $\sL_i$ for each variable $x_i$ with $i\in \{1,\dots,k\}$ such that for each monomial $W_j$ of $W$, we have $W_j(\sL_1,\dots, \sL_k) \cong \omega_{\\log,\sC}$.  But this can be described more cleanly using a reformulation in terms of principal bundles due to Polishchuk-Vaintrob \cite[\S 3.2]{PoVa:11}.

We will use the integer weights $c_1,\dots, c_N,d$ for $W$ with each $q_i= c_i/d$, as described in Remark~\ref{rem:integerweights}. Let $G\subset \aut(W)$ be an admissible group, and let $\Gamma$ be the
subgroup of $(\CC^*)^N$ generated by $G$ and
$\CC^*_R=\{(\lambda^{c_1}, \dots, \lambda^{c_N}) \mid \lambda\in
\CC^*\}$, where this second factor $\CC^*_R$ corresponds to  the
quasihomogeneity of $W$.  It is easy to see that
\begin{equation}\label{eq:GintersectCR}
G\cap \CC^*_R = \genj .
\end{equation}
We define a surjective homomorphism
  $$\chiR: \Gamma\rightarrow \CC^*$$ by sending $G$ to $1$ and
$(\lambda^{c_1}, \dots, \lambda^{c_N})$ to
$\lambda^d$. Equation~\eqref{eq:GintersectCR} shows that the map $\chiR$
is well-defined and that $\ker(\chiR)=G$.  Let $\pklogc$ denote the
principal $\CC^*$-bundle associated to $\klogc$.
   
\begin{defn}
  A $\Gamma$-structure on an orbicurve $\Ccal$ is
\begin{enumerate}
\item A principal $\Gamma$-bundle $\Pcal$ on $\Ccal$ such that the
  corresponding map $\Ccal \to \B\Gamma$ to the classifying stack
  $\B\Gamma = [pt/\Gamma]$ is representable.
\item A choice of isomorphism $\spn:(\chiR)_*\Pcal \cong
  \pklogc$.  Here $(\chiR)_*\Pcal$ denotes the principal $\CC^*$
  bundle on $\Ccal$ induced from $\Pcal$ by the map $\chiR$.
\end{enumerate}
\end{defn}
An equivalent way to state (2) is to recognize that the homomorphism
$\chiR$ induces a morphism of stacks $\B\chiR:\B\Gamma \to \B \CC^*$
and (2) is equivalent to the requirement that the composition $\B\chiR
\circ \Pcal:\Ccal \to \B\CC^*$ be equal to the morphism of stacks
$\Ccal \to \B\CC^*$ induced by the principal $\C^*$-bundle $\pklogc$.

Given a $\Gamma$-structure on $\Ccal$, the projection $\pi_i:\Gamma \subseteq (\CC^*)^N \to \CC^*$ to the $i$-th
factor for each $i\in \{1,\dots N\}$ defines a collection of line
bundles 
\[\Lcal_1, \cdots, \Lcal_N.\]
 It is easy to check that
$\pi^d_i=\chiR^{c_i}$, and thus we have
$   \Lcal^d_i=\omega^{c_i}_{\Ccal, \log}$.   

Let $W=\sum_j \alpha_j W_j$, where each $W_j$ is a monomial and $\alpha_j\in \C$ is a coefficient.  The monomial $W_j$ induces a homomorphism $(\CC^*)^N\rightarrow \CC^*$. Because $W$ is $G$-invariant, we have $W_j|_G=1$. Therefore, $W_j$ defines a homomorphism $W_j:
\Gamma/G\rightarrow \CC^*$.  

A straightforward check of how  $W_j$ acts on the subgroup
$\CC_R^* = \{(\lambda^{c_1}, \dots, \lambda^{c_N})\}$ shows that this homomorphism is an isomorphism. Hence we have $W_j(\pi_1,
\cdots, \pi_N)=\chiR$. This implies that for each monomial $W_j$ we have 
\[
W_j(\Lcal_1, \dots, \Lcal_N)\cong\omega_{\Ccal, \log}.
\]

Therefore, associated to each $\Gamma$-structure on a pointed stable orbicurve $\Ccal$, there is a collection of line bundles $\Lcal_1,\dots, \Lcal_n$ on $\Ccal$ and isomorphisms 

\begin{equation} \label{eq:wspin}
W_j(\Lcal_1, \dots, \Lcal_N) \cong \omega_{\Ccal, \log}.
\end{equation}
These line bundles play an important role in the construction of the virtual cycle.
The original definition of the moduli problem defining $\sW^{W,G}_{g,k}$ was given in terms of a collection of line bundles satisfying Equation~\eqref{eq:wspin}. 

\begin{defn}
Let $\sW_{g,k}^{W,G}$ be the stack of tuples $(\Ccal,p_1,\dots,p_k,\Pcal, \spn)$, where $\Ccal,p_1,\dots, p_k$ is a stable $k$-pointed orbicurve of genus $g$ and $(\Pcal,\spn)$ is a $\Gamma$-structure on $\Ccal$. 
\end{defn}
\begin{remark}
The stack $\sW_{g,k}^{W,G}$ depends only on $g$, $k$, $G$, and $\Gamma$ and is completely independent of the polynomial $W$.
\end{remark}

We have several natural morphisms of  $\sW_{g,k}^{W,G}$.  The first of these is the forgetful morphism $st:\sW_{g,k}^{W,G} \to \mgkbar$ (analogous to the stabilization morphism of Gromov-Witten theory.
The following result is fundamental. 
\begin{thm}[{\cite[Thm 2.2.6]{FJR}}] 
If $J\in G\subseteq \Gmax$, then for any nonnegative integers $g$ and $k$ with $2g-2+k>0$, the morphism $st:\sW_{g,k}^{W,G} \to \mgkbar$  is proper and quasifinite (but not representable).
\end{thm}

\begin{example}
If $g=0$ and $k=1$, then for any admissible group $G$, for each choice of orbifold three-pointed genus-zero curve, there is at most one $\Gamma$-structure on that curve.  Moreover, the automorphisms of the $\Gamma$-structure are exactly $G$.  So for any particular choice of $\bgamma = (\gamma_1,\gamma_2,\gamma_3)$ the stack
$\sW^{W,G}_{0,3}(\bgamma)$ is either empty or is isomorphic to $\B G$.
\end{example}

\begin{remark}
If the grading element $J$ is not contained in $G$, it is not hard to show that the stack $\sW_{g,k}^{W,G}$ is not proper for general $g$ and $k$. 
\end{remark}

We also have an analogue of the evaluation morphism.
Let $G_{p_i}$ be the local group of $\Ccal$ at the marked point $p_i$.  Since we are working over $\C$, it has a canonical generator.  Let $\gamma_{p_i}$ denote this canonical generator.
 The morphism $\Pcal:\Ccal \to \B\Gamma$ implies that each $G_{p_i}$ has a
homomorphism to $\Gamma$.   
The fact that $\omega_{\Ccal, \log}$ has no orbifold structure at any marked point implies that
$G_{p_i}$ actually maps to $\ker(\chiR)=G\subset \Gamma$. And
representability of the morphism $\Pcal: \Ccal \to \B\Gamma$ implies that the
map $G_{p_i} \to \ker(\chiR)=G$ is injective.

Thus for each $i\in \{1,\dots, k\}$ we have an ``evaluation'' map $ev_i$ mapping  $\sW_{g,k}^{W,G}$ to $G$, or rather to the inertia stack $I[\C^N/G] = [\C^N/G] \sqcup \coprod_{\gamma\neq1 \in G} BG$.  This map is given by sending $(\Ccal,p_1,\dots, p_k, \Pcal, \spn)$ to $[0/G]\subset [\C^N/G]$ if $\gamma_{p_i} = 1$, and otherwise sending it to the copy of $\B G$ corresponding to the image of $\gamma_{p_i}$ in $G$.  

This gives us a decomposition of $\sW_{g,k}^{W,G}$ into open and closed substacks
\[
\sW_{g,k}^{W,G} = \coprod_{\gamma_1,\dots,\gamma_k} \sW_{g,k}^{W,G}(\gamma_1,\dots, \gamma_k),
\]
where each $\sW_{g,k}^{W,G}(\gamma_1,\dots,\gamma_k)$ is the locus where the $i$th evaluation map lies in the copy of $\B G$ corresponding to the element $\gamma_i \in G$.

Finally we have a ``forgetting tails'' morphism.  If the canonical generator of $G_{p_i}$ maps to $J\in G$, then it is not hard to show that we can ``desingularize'' the line bundles $(\chiR)_*\Pcal$ and $\pklogc$ by forgetting the orbifold structure at $p_i$, and in a neighborhood of $p_i$ the new line bundles are isomorphic to $\pkk_{\Ccal}$ (see \cite[\S2.2.3]{FJR} for details).
Thus, if we start with a $k+1$-pointed curve marked with $J$ at the $i$th point, then forgetting the marked point $p_i$ means the homomorphism $\spn$ maps $\Pcal$ to the log-canonical bundle of the corresponding $k$-pointed curve.  This gives us a morphism 
\[
\sW_{g,k+1}^{W,G}(\gamma_1,\dots, J, \dots, \gamma_k) \to \sW_{g,k}^{W,G}(\gamma_1, \dots, \gamma_k), 
\]
which we call \emph{forgetting tails}.  Note that we can only forget tails which are marked with the element $J$.  Forgetting any other tail will fail to produce a $\Gamma$-structure on the resulting curve.

\subsection{Virtual cycle}\label{sec_virtcycle}

The virtual cycle of FJRW theory is a homology class $\left[\sW_{g.k}^{W,G}(\gamma_1,\dots,\gamma_k)\right]^{virt}$ on $\sW^{W,G}_{g,k}(\gamma_1,\dots,\gamma_k)$ satisfying certain key properties or ``axioms,'' which we will describe below.

There are several definitions of the virtual class for FJRW theory.    
The original idea for the virtual class in the case where $W = x^r$ ($r$-spin curves) is due to Witten \cite{Wit:93}.  An algebraic construction for $r$-spin curves was given by Polishchuk-Vaintrob \cite{PoVa:01} and a K-theoretic construction given by Chiodo \cite{Chiodo:06a,Chiodo:06b}.  An analytic construction for the $r$-spin virtual class based on Witten's original outline was given by T.~Mochizuki \cite{Moc:06}.

In the case of a general polynomial and group, an analytic construction of the virtual class was given in \cite{FJR-ip-WEVFC} using the first-order elliptic PDE 
       \begin{equation}\label{eq:WittenEq}
       \bar{\partial} s_i+\overline{\frac{\partial W}{\bar{\partial}
    s_i}}=0,\end{equation}
    called the \emph{Witten Equation}, where each $s_i$ is a $C^{\infty}$-section of the $i$th line bundle $\Lcal_i$ of our $\Gamma$-structure.

An algebraic definition of the virtual class was given by Polishchuk-Vaintrob in \cite{PoVa:11} using graded matrix factorizations. And more recently Chang-Li-Li \cite{CLL} 
used the cosection localization technique of Chang-Kiem-Li \cite{KiemLi, ChangLi}
to construct the virtual cycle on the stacks $\sW_{g,k}^{W,G}(\gamma_1, \dots, \gamma_k)$ where every $\gamma_i$ is narrow (see Remark~\ref{def:narrow}).  They prove that their construction agrees with both the analytic construction of Fan-Jarvis-Ruan and the algebraic construction of Polishchuk-Vaintrob in the narrow case.  

Polishchuk-Vaintrob showed that the algebraic and analytic virtual cycles agree for all simple (ADE) singularities (these are polynomials with central charge $\hat{c}<1$).  Gu\'er\'e has also shown that for chain-type polynomials all these constructions of the virtual class coincide, up to a rescaling of the broad sectors \cite[Thm 3.24]{Guere}.
Unfortunately, it is not yet known whether the algebraic and analytic constructions agree with arbitrary broad insertions for more general polynomials. 

Pushing the virtual cycle down to $\mgkbar$ gives the FJRW CohFT.  
\begin{defn}
For each $i\in \{1,\dots, k\}$ let $\alpha_i \in \sH_{\gamma_i}$. 
We define $\Lambda_{g,k}^{W,G}(\alpha_1,\dots, \alpha_k) \in H^*(\mgkbar)$ as
\begin{equation}\label{eq_Lambda}
 \Lambda_{g,k}^{W,G}(\alpha_1,\dots,\alpha_k) = \frac{|G|^g}{\deg st} PD st_*\left(\left[\sW^{W.G}_{g,k}(\gamma_1,\dots,\gamma_k)\right]^{virt}\cap \prod_{i=1}^k ev_i^* \alpha_i \right).
\end{equation}
This definition is extended linearly to all choices of $k$-tuples $\alpha_1,\dots, \alpha_k \in \sH_{W,G}$.
\end{defn}

\begin{thm}\cite{FJR}
The collection of cohomology classes given by $\{\Lambda_{g,k}^{W,G}\}$ as defined  in Equation~\eqref{eq_Lambda} gives a CohFT with flat identity element equal to $\one\in \sH_J$, and satisfies the following additional properties:
\begin{enumerate}
%%%%%%%%This is already in the CohFT axioms
%\item \textbf{Forgetting tails:} \label{prop_first} If $\tau:\bMgn{g}{k+1} \to \mgkbar$ is the usual forgetting tails map for stable curves, and $\one\in \sH_J$ is the identity element then 
%\[\Lambda^{W,G}_{g,k}(\alpha_1, \ldots, \alpha_{k-1},\one) = \tau^* \Lambda^{W,G}_{g,k}(\alpha_1, \ldots, \alpha_{k-1})\]
%and 
%\[
%\Lambda^{W,G}_{0,3}(\alpha_1, \alpha_2,\one) = \langle \alpha_1,\alpha_2\rangle,
%\]
%where $\langle\ ,\ \rangle\sH_{W,G}\otimes \sH_{W,G} \to \C$ is the pairing defined in Section~\ref{sec_pairing}.
 
\item \textbf{Dimension:} \label{prop_dimension}
If $\alpha_i \in \sH_{\gamma_i}$ for each $i\in \{1,\dots, k\}$, then 
the dimension of $\Lambda^{W,G}_{g,k}(\alpha_1,\dots, \alpha_k)$ is given by
\[
dim(\Lambda^{W,G}_{g,k}(\alpha_1,\dots, \alpha_k)) = (\hat{c} - 3)(1-g) + k - \sum_{i = 1}^k \deg_\C \alpha_i.
\]

\item \textbf{Integral Degrees:}
If $\alpha_i \in \sH_{\gamma_i}$ for each $i\in \{1,\dots, k\}$, and if
 the degree of the desingularization $|\sL_i|$ of each line bundle $\sL_i$ defined by the $\Gamma$-structure on $\sW^{W,G}_{g,k}(\gamma_1,\dots, \gamma_k)$ is not integral, then $\Lambda^{W,G}_{g,k}(\alpha_1,\dots, \alpha_k) = 0$.

\item \textbf{$\Gmax$ Invariance:} 
Letting the group $\Gmax$ act on $\sH_{W,G}$ in the obvious way, the 
class $\Lambda_{g,k}^{W,G}(\alpha_1,\dots, \alpha_k)$ is $\Gmax$-invariant.

\item \textbf{Disjoint Sums:}
\label{ax_sumsing}
If $W_1 \in \C[x_1,\dots, x_n]$ and $W_2\in \C[y_1,\dots, y_m]$ are each nondegenerate quasihomogeneous polynomials with $G_1$ admissible for $W_1$ and $G_2$ admissible for $W_2$, and if $\alpha_1,\dots, \alpha_k \in \sH_{W_1,G_1}$ and $\beta_1,\dots, \beta_k \in \sH_{W_2,G_2}$,
then $\sH_{W_1+W_2,G_1\times G_2} = \sH_{W_1,G_1}\otimes \sH_{W_2,G_2}$ and 
\[
\Lambda^{W_1 + W_2, G_1 \times G_2}(\alpha_1\otimes\beta_1,\dots, \alpha_k\otimes\beta_k) = \Lambda^{W_1, G_1}_{g,k}(\alpha_1,\dots,\alpha_k)  \Lambda^{W_2,G_2}_{g,k}(\beta_1,\dots,\beta_k).
\]

\item \textbf{Deformation Invariance:}
\label{ax_definv}
Let $W_t$ be a one-parameter family (one real parameter) of non-degenerate polynomials of a given weight $(q_1,\dots, q_N)$, and let $G$ be a symmetry group which is admissible for each $W_t$. The family $W_t$ defines an isomorphism of the state spaces $\sH_{W_t,G}$ to $\sH_{W_{t_0},G}$ for any $t_0$.  If $\alpha_1,\dots, \alpha_k \in \sH_{W_{t_0},G}$, then (using the corresponding family of isomorphisms) the class $\Lambda^{W_t,G}_{g,k}(\alpha_1,\dots, \alpha_k)$ is independent of $t$.

\item \textbf{Topological Euler Class:}  Let $\sL_1,\dots,\sL_k$ be the line bundles induced by the universal $\Gamma$-structure on the universal curve $\pi:\Ccal\to \sW^{W,G}_{g,k}(\gamma_1,\dots,\gamma_k)$.

If each $\gamma_i$ is narrow (i.e., the fixed locus of $\gamma_i$ is $0\in \C^N$) for each $i\in \{1,\dots,k\}$, then we can define a particular map (the \emph{Witten map}) $\mathscr{D}:\pi_*(\Lcal_1\oplus \cdots\oplus\Lcal_k) \to R^1\pi_*(\Lcal_1\oplus \cdots\oplus\Lcal_k)$ in the derived category of $\sW^{W,G}_{g,k}(\gamma_1,\dots,\gamma_k)$
such that the  
virtual class $[\sW^{W,G}_{g,k}(\gamma_1,\dots,\gamma_k)]^{virt}$ is Poincar\'e dual to a topological Euler class for $\mathscr{D}$.
This implies the following additional properties hold:
	\begin{enumerate}
    \item {\bf Concavity:}\label{ax:convex}
    Suppose that each $\gamma_i$ is narrow.  If
$\pi_*\left(\bigoplus_{i=1}^t\Lcal_i\right)=0$, then the virtual
cycle is given by capping the top Chern class of  $\left(R^1 \pi_* \left(\bigoplus_{i=1}^t\Lcal_i\right)\right)^*$ 
with the usual fundamental cycle of the moduli space:
\[
\left[ \sW^{W,G}_{g,k}(\gamma_1,\dots, \gamma_k) \right]^{vir} = 
c_{top}\left(\left(R^1\pi_*\bigoplus_{i=1}^t\Lcal_i \right)^*\right) \cap \left[\sW^{W,G}_{g,k}(\gamma_1,\dots, \gamma_k)\right]
\]
\item   {\bf  Index zero:} \label{ax:wittenmap} Suppose that  $\dim (\sW^{W,G}_{g,k}(\gamma_1,\dots, \gamma_k))=0$
and each $\gamma_i$ is narrow.

If the pushforwards $\pi_* \left(\bigoplus\Lcal_i\right)$ and $R^1\pi_*
\left(\bigoplus \Lcal_i\right)$ are both vector bundles of the same
rank, then the virtual cycle is just the degree of the
Witten map times the fundamental cycle.
	\end{enumerate}
\end{enumerate}
\end{thm}

The dimension, integral degrees, and $\Gmax$ invariance properties all give selection rules that force many of these classes to vanish.  The other main tool for computing these classes is the Concavity Axiom, combined with reconstruction rules arising from the WDVV equation (see \cite{Francis} for many examples of how this works).  J\'er\'emy Gu\'er\'e has also been able to use the Polishchuk-Vaintrob matrix-factorization approach to compute the virtual classes in several cases where the insertions are narrow but not concave \cite{Guere}.

\section{Brief Survey of Landau-Ginzburg Mirror Symmetry}\label{sec:mirror}

For each invertible polynomial $W$ and each group $G\subseteq G_{\max,W}$ of diagonal symmetries of $W$, the Berlund-H\"ubsch-Henningson-Krawitz (BHHK) mirror construction produces  a ``transpose" polynomial $W^T$ and a ``transpose'' group $G^T$, where $W^T$ is again an invertible polynomial, and $G^T$ is a diagonal group of symmetries of $W^T$.

The construction of the transpose is very simple.  First, since the polynomial $W$ is invertible, we may rescale the variables so that $W$ can be written as a sum of monomials all with coefficient $1$.
\[
W =\sum_{i=1}^N \prod_{j=1}^N x_j^{a_{ij}}.
\]
Let the matrix $E = (a_{ij})$ have the exponents of the various monomials for its entries.  It is not hard to show that this matrix is invertible when the polynomial $W$ is invertible.

The BHHK mirror $W^T$ is the polynomial corresponding to the exponent matrix $E^T$:
\[
W^T = \sum_{i=1}^N \prod_{j=1}^N x_j^{a_{ji}}.
\]
Now, the group $\Gmax$ is isomorphic to the additive group 
\[
\{\g = (g_1,\dots, g_N)^T \in (\Q/\Z)^N \mid E\g \in \Z^N \},
\]
 and we may write any subgroup $G$ in additive notation. 
In this additive notation, the BHHK mirror group $G^T$ is the group 
\[
G^T = \{ \h \in (\Q/\Z)^N \mid \h^T E \g \in \Z, \forall \g\in G\}.
\]
The following properties of $G^T$ are easy to verify:
\begin{enumerate}
\item If $G\subseteq G_{\max,W}$ then $G^T\subseteq G_{\max,W}$.
\item $G\subseteq G_{\max,W}$ contains $J$ if and only if $G^T \subset G_{\max,W^T} \cap SL(N)$.
\item $G_{\max,W}^T = \{0\}$.
\item For any subgroups $G_1 \subseteq G_2\subseteq G_{\max,W}$ we have $G_2^T\subseteq G_1^T$ and 
\[G_2/G_1 \cong G_1^T/G_2^T.\]
\end{enumerate}

The guiding conjecture for Landau-Ginzburg mirror symmetry is the following.
\begin{conj}[Landau-Ginzburg Mirror Conjecture]
If $W$ is an invertible polynomial and $G$ is an admissible group of symmetries of $W$, then there should be a Landau-Ginzburg $B$-model associated to $W^T$ and $G^T$ such that the FJRW theory (A model) of $W,G$ is isomorphic to the Landau-Ginzburg B model of $W^T,G^T$.
\end{conj}

\subsection{Mirror Symmetry with Saito-Givental}
The first step in mirror symmetry is to identify the appropriate B model.  In the case that $G = G_{\max,W}$ the dual group $G^T$ is trivial, and the Landau-Ginzburg B model is ``unorbifolded.''  In this case the B model can be constructed from  Saito-Givental theory.
To any quasihomogeneous polynomial $f$ with an isolated singularity at the origin, we can consider the local algebra (a.k.a. Milnor ring) $\sQ_f$ of $f$.  This is a Frobenius algebra with the residue paring.  

Given a choice of primitive form, K.~Saito \cite{Sai1,Sai2,Sai3,SaiTak}
has constructed a semi-simple Frobenius manifold from $f$ that agrees with $\sQ_f$ at the origin.  Since this Frobenius manifold is generically semisimple, one can (with some work) use Givental's theory of higher-genus Frobenius manifolds \cite{Giv:quantization,givental}
 to construct a potential function that behaves as if it came from a CohFT.

For more general $G$, the construction of the B model is not yet entirely known.
Kaufmann and Krawitz \cite{Kr,kau1, kau2,kau3},
following ideas of Intriligator-Vafa \cite{IV} 
showed how to construct an ``orbifolded'' Milnor ring $[\sQ_{W^T}/G^T]$ that is  a Frobenius algebra.  In a few special cases candidate orbifolded Frobenius manifolds have been constructed \cite{BasTak, BLee},
but there is still no general construction of an appropriate B model orbifolded Frobenius manifold.

The Landau-Ginzburg Mirror Symmetry conjecture was proved for the $A_n$ singularity (corresponding to the polynomial $W = x^{n+1}$ with group $\Gmax = \genj$) in genus $0$ in \cite{wit:93} (see also \cite{jkv}), 
in genus $1$ and $2$ by Y.P. Lee \cite{Lee}, 
and in higher genus by \cite{FSZ}. 
The conjecture was proved for all remaining simple singularities in \cite{FJR}.  For all the simple singularities except $D_n$ with $n$ even, the only admissible group is the maximal group $\Gmax$, and all primitive forms differ only by scalar multiplication, so for these singularities, there is no orbifolding on the B side and there is a clear choice of Saito-Givental construction to play the role of the B model.  

However, in the case of $D_n$ with $n$ even, the group $\genj$ has index $2$ inside of $\Gmax$, and so the B side is should be orbifolded by $\Z/2$. 
 In this special case, it turns out that the orbifolded Frobenius algebra on the B side is isomorphic to the unorbifolded Frobenius algebra for $D_n$ again, so one might expect that in this case the entire orbifolded B model should be isomorphic to the unorbifolded B model for $D_n$, and indeed \cite{FJR,FFJMR} shows that the FRJW theory of $D_n$ with the group $\genj$ (and $n$ even) is isomorphic to the unorbifolded Saito-Givental B model for $D_n$.

For more general polynomials, whenever $G = \Gmax$ then $G^T = \{0\}$. In this case LG mirror symmetry is completely proved.  First, Krawitz \cite{Krthesis} 
proved the conjecture for Frobenius algebras (the case of $g=0$ and $k=3$), 
But for Frobenius manifolds ($g=0$ and all $k\ge 3$), the problem is complicated by the fact that one must identify the correct primitive form to define the Saito Frobenius manifold on the B side.  This was done and the Mirror Conjecture proved for simple elliptic singularities by Krawitz, Milanov and Shen \cite{ShKr, MiSh:12} (see also \cite{SaTa}).  The Mirror Conjecture was proved for Arnold's 14 exceptional unimodular singularities by Li-Li-Saito-Shen \cite{LLSS,SaitoMS}. Finally, He-Li-Shen-Webb \cite{HLSW} 
have identified the appropriate primitive form for the Saito theory in general and have proved that Landau-Ginzburg Mirror Symmetry holds for arbitrary invertible polynomials with $G=\Gmax$.  Since Saito's Frobenius manifolds are generically semisimple, one can then use Telemann's theorem \cite{Tel:07} to show that Givental's higher genus construction also agrees with the FJRW theory.

In the orbifolded case (when $G\neq \Gmax$) it is proved in \cite{FJJS} that the Landau-Ginzburg Mirror Symmetry conjecture holds at the level of Frobenius algebras for a large class of invertible polynomials and groups, but for certain chain-type polynomials this conjecture remains open.  As mentioned above, before we can even hope to prove an orbifolded mirror symmetry result at the level of Frobenius manifolds, we will need a general construction of an orbifolded version of Saito's Frobenius manifold.

Note also that, unlike with the (unorbifolded) Saito Frobenius manifolds, we have no reason to expect that these orbifolded Frobenius manifolds will be semisimple (or even generically semisimple), which means that proving mirror symmetry in genus zero will not automatically give the result in higher genera.

\subsection{Integrable Hierarchies}

The original Landau-Ginzburg mirror symmetry conjecture was the so-called ``Generalized Witten Conjecture'' \cite{Wit:93} that suggested the potential function of $r$-spin theory (FJRW theory for the polynomial $x^r$) should satisfy the KdVr (a.k.a.~Gelfand-Dikii) hierarchy. 
 This was proved by Witten in genus $0$ \cite \cite{Wit:93} (see also \cite{jkv}), by Lee \cite{Lee} for $g=1,2$, and by Faber-Shadrin-Zvonkine \cite{FSZ} for higher genera.  It can be thought of as an alternative form of mirror symmetry for the $A_r$ singularity.

In fact, there are several ways to associate an integrable hierarchy to each of the simple singularities \cite{DriSok,KacWaki,Giv:nkdv,GiMi:05} but these all turn out to be equivalent \cite{HoMi:93, FGM, LWZ}. 
In \cite{FJR,FFJMR} it is shown that the potential function of the FJRW theory for each of the simple singularities satisfies the corresponding (transpose) integrable hierarchy.

Dubrovin and Zhang \cite{DuZh} (see also \cite{BPS,Buryak})
have identified a way to construct an integrable hierarchy associated to any CohFT, and it is natural to conjecture that for any invertible polynomial, the potential function of FJRW theory should satisfy the corresponding (transpose) integrable hierarchy.  

\subsection{$I$- and $J$-functions}

A third way to think about mirror symmetry is in terms of Givental's $I$- and $J$-functions. Specifically one can construct both the $J$-function of FJRW theory and the $I$-function arising from certain period integrals at the Gepner point (see, for example, \cite{HKQ}).  The mirror conjecture in this setting is that there is an explicit change of variables (a ``mirror map'') which identifies the $I$-function with the $J$-function.

J\'er\'emy Gu\'er\'e has proved this form of mirror symmetry holds for all chain-type polynomials with group $\Gmax$ 
\cite[Thm 1.15]{Guere}.

Several others have also identified the FJRW $I$- and $J$-functions for various cases of polynomials and groups as a step on the way to proving various cases of the LG/CY correspondence.   In fact, this is currently the most common approach to verifying specific cases of the LG/CY correspondence.  We will discuss these further in the next section.

\section{Landau-Ginzburg/Calabi-Yau Correspondence}\label{sec:LG-CY}

As mentioned in the introduction, a quasihomogeneous polynomial $W$ with an isolated singularity at
the origin defines a smooth hypersurface $X_W = \{W=0\}$ in weighted
projective space, and if $G$ is an admissible group of automorphisms of $W$, there is an induced
action of the group $\tilde{G} = G/\genj$ on $X_W$.  

The hypersurface $X_W$ is Calabi-Yau precisely when the sum $\sum_{i=1}^N q_i$ of the charges is equal to $1$.  The quotient orbifold  $[X_W/{\tilde{G}}]$ is Calabi-Yau when also $G \subseteq SL(N)$.

The
\emph{Calabi-Yau/Landau-Ginzburg (CY/LG) Correspondence} predicts that when $[X_W/{\tilde{G}}]$ is Calabi-Yau, then the FJRW theory of $W,G$ should agree in all genera with the Gromov-Witten theory of $[X_W/\tilde{G}]$ (after analytic continuation and a symplectic transformation).

This conjecture is interesting for its own sake, but it is also important because FJRW invariants are generally believed to be easier to compute than Gromov-Witten invariants (see \cite{Francis, Guere} for examples).  So the LG/CY correspondence could provide new tools for computing Gromov-Witten invariants.

The first result on the LG/CY correspondence is that of \cite{ChiRu:09}, who proved that the Gromov-Witten state space $H^*_{CR}([X_W/\tilde{G}])$ and the FJRW state space $\sH_{W,G}$ match for all $W$ satisfying the Calabi-Yau condition $\sum q_i = 1$ and for all admissible $G$ (not only those in $SL(N)$).  They further showed that the narrow sectors of the FJRW state space correspond to the ambient cohomology of $H^*_{CR}([X_W/\tilde{G}])$.

The LG/CY correspondence has been verified in genus zero for the Fermat Quintic with $\tilde{G} = \{1\}$ \cite{ChiRu:10}; for the narrow/ambient part of a general hypersurface in Gorenstein weighted projective space with $\tilde{G} = \{1\}$ 
% (which is the G-max invariant part)  def invariance means that they can use fermat in the Gorenstein case.
\cite{CIR:11}; 
for the mirror quintic (with group $\tilde{G} = (\Z/5)^3$) \cite{PrSh:13,LeSh}; and for arbitrary Fermat polynomials with arbitrary admissible groups \cite{PLS:14}. This last paper \cite{PLS:14} is also interesting because it connects the LG/CY correspondence to the better-understood Crepant Transformation Conjecture.
 
The LG/CY correspondence has been verified in all genera for elliptic curves  with $G = \Gmax$ \cite{ShKr, MiRu, MiSh:12}.  

P.~Acosta in \cite{Acosta:14} has recently shown how to generalize the correspondence to Fano and general-type hypersurfaces by replacing analytic continuation with asympototic expansions.  He has proved an LG/Fano and an LG/General-Type correspondence in genus zero for arbitrary nondegenerate polynomials with group $\tilde{G} = \{1\}$.

\section{Gauged Linear Sigma Model}

E.~Clader in \cite{Clader} has generalized FJRW theory from hypersurfaces to certain Calabi-Yau complete intersections with a \emph{hybrid model}, and she has shown that the LG/CY correspondence holds for these complete intersections.    It is natural to ask whether there are further generalizations of the LG/CY correspondence to more general complete intersections, toric varieties, or even to GIT or symplectic  quotients by nonabelian groups.

The answer is ``yes,'' and physically it is given by Witten's \emph{Gauged Linear Sigma Model (GLSM)}, which he developed in the early 1990s \cite{Wit:92a}.  His physical arguments, when translated in to a mathematical setting, should  extend (and allow us to prove) the LG/CY correspondence to complete intersections, toric varieties, and more general quotients by nonabelian groups.

From the point of view of partial differential equations, the gauged linear
sigma model generalizes the Witten Equation \eqref{eq:WittenEq} to the
\emph{Gauged Witten Equation}
\begin{align*}%\label{eq:GaugedWitten}
\bar{\partial}_A u_i+\overline{\frac{\partial W}{\partial
    u_i}}&=0,\\ *F_A&=\mu,
\end{align*}
where $A$ is a connection of certain principal bundle, and $\mu$ is the moment map of a certain Hamiltonian group action.  Both the
Gromov-Witten theory of a Calabi-Yau complete intersection and the
corresponding ``dual'' Landau-Ginzburg theory are generally expressible as gauged linear sigma models. Furthermore, the LG/CY correspondence can be interpreted as a
variation of the moment map $\mu$ in the GLSM. 

In the special case where $W=0$, the Gauged Witten Equation is called the \emph{symplectic vortex equation}, and it has been widely studied, both from an algebraic and a symplectic point of view.  A very important development in this direction is the theory of stable quotients \cite{MOP:11} and stable quasimaps \cite{CFKM:11,CFKi:10, Kim:11,CCFK:14}. 

Recently Fan-Jarvis-Ruan in \cite{FJR:GLSM} have described a mathematical approach to the GLSM.  Their construction is essentially a union of
FJRW-theory (in the Polishchuk-Vaintrob formulation) with stable quasimaps. The relation between the GLSM and FJRW-theory can be viewed as a generalization from a finite Abelian gauge group $G$
(FJRW-theory) to an arbitrary reductive group (GLSM). 

Just as in FJRW theory, the GLSM has both {narrow} and {broad} sectors (see Remark~\ref{def:narrow}) in its state space, and the theory for narrow sectors admits a purely algebraic construction of the virtual cycle, using the cosection localization techniques of \cite{CLL, KiemLi, ChangLi}. 
As special cases, one recovers Clader's hybrid model and Chang-Li's stable maps with $p$-fields \cite{ChangLi}.
In the broad case, cosection localization does not apply, but there is an analytic construction of the virtual cycle \cite{FJR:GLSM2} (see also \cite{TX}).

\bibliography{references}{}
\bibliographystyle{halpha} %%use halpha instead of alpha, because it does a better job with arXiv refs.

\end{document}